\documentclass[12pt]{amsart}
\usepackage[utf8]{inputenc}
\usepackage[margin=1in]{geometry}
\usepackage{amsmath}
\usepackage{tikz-cd} 
\usepackage{enumitem} 
\usepackage{amsfonts}
\usepackage{amssymb}
\usepackage{amsthm}
\usepackage{xcolor}
\usepackage{bm}
\usepackage{mathtools}
\usepackage{verbatim}
\usepackage{pinlabel}
\usepackage{caption}
\usepackage{subcaption}
\usepackage{graphicx}
\usepackage{hyperref}
\usepackage{float}

\usepackage{microtype}
\usepackage{mathrsfs}
\usetikzlibrary{topaths}
\usetikzlibrary{calc}

\newtheorem{theorem}{Theorem}
\newtheorem{problem}[theorem]{Problem}
\newtheorem{lemma}[theorem]{Lemma}

\newtheorem{prop}[theorem]{Proposition}
\newtheorem{cor}[theorem]{Corollary}

\newtheorem{definition}[theorem]{Definition}

\DeclareMathOperator{\id}{id}

\newcommand{\nl}{\mathrm{(NL)}}
\newcommand{\acts}{\curvearrowright}
\newcommand{\e}{\epsilon}

\newcommand{\zz}{\ensuremath{\mathbb{Z}}}

\newcommand{\nn}{\ensuremath{\mathbb{N}}}
\newcommand{\hh}{\ensuremath{\mathbb{H}}}

\newcommand{\Ga}{\Gamma}
\newcommand{\op}{\operatorname}

\newcounter{scomments}

 \newcounter{gcomments}

\newcounter{racomments}

 \newcounter{rocomments}

\makeatletter
\let\@wraptoccontribs\wraptoccontribs
\makeatother

\begin{document}

\title{Property $\boldsymbol{\nl}$ in Coxeter groups}
\date{\today}


\author[S.~Balasubramanya]{Sahana Balasubramanya}
\address{Department of Mathematics, Lafayette College, USA}
\email{hassanba@lafayette.edu}

\author[G.~Burkhalter]{Georgia Burkhalter\textsuperscript{$\dagger$}}
\address{University of North Georgia, USA}
\email{georgiaburkhalter@gmail.com}
\thanks{\textsuperscript{$\dagger$} denotes undergraduate authors.}

\author[R.~Niebler]{Rachel Niebler\textsuperscript{$\dagger$}}
\address{Haverford College, USA}
\email{rlpniebler@gmail.com}

\author[R.~Shapiro]{Roberta Shapiro}
\address{Department of Mathematics, Georgia Institute of Technology, USA}
\email{shapirorh@gmail.com}
\maketitle


\begin{abstract}
    A group has Property $\nl$ if it does not admit a loxodromic element in any hyperbolic action. In other words, a group with this property is inaccessible for study from the perspective of hyperbolic actions. This property was introduced by Balasubramanya, Fournier-Facio and Genevois, who initiated the study of this property. We expand on this research by studying Property $\nl$ in Coxeter groups, a class of groups that are defined by an underlying graph. One of our main results show that a right-angled Coxeter group (RACG) has Property $\nl$ if and only if its defining graph is complete. We then move beyond the right-angled case to show that if a defining graph is disconnected, its corresponding Coxeter group does not have Property $\nl$. Lastly, we classify which triangle groups (Coxeter groups with three generators) have Property $\nl$. 
\end{abstract}


\section{Introduction}\label{sec:intro}

Taking inspiration from the work of Gromov \cite{Gromov}, isometric group actions on hyperbolic spaces (sometimes called \emph{hyperbolic actions}) have become a popular avenue of research in geometric group theory. Understanding the ways in which a group can act on a hyperbolic space can lend insight into its group-theoretic properties.

Property $\nl$, first formally defined in \cite{PropNL}, attempts to understand groups that are inaccessible from this perspective. A group $G$  has Property $\nl$ if no element acts loxodromically for any action of $G$ on any hyperbolic space (see Sections~\ref{subsec:isom} and ~\ref{subsec:propnl} for details). 

Coxeter groups appear naturally as abstractions of reflection groups and provide information about the symmetries of an associated space. Such groups are most commonly defined using a \emph{Coxeter graph}, which is an edge-weighted graph where the vertices correspond to generators of order 2, while vertices $a$ and $b$ are connected by an edge of weight $n \in \nn$ if and only if $n$ is the minimal natural number such that $(ab)^n=1$ (see Section~\ref{subsec:cox} for details). If all edge weights equal two, the group is called \emph{right-angled}. In this paper, we are motivated by the following question. 

\begin{problem}\label{ques1} Which Coxeter groups have Property $\nl$? Moreover, how can this be deduced from the defining graph of a Coxeter group? 
\end{problem}

We provide a partial answer to this question by proving the following theorems. 

\begin{theorem}\label{thm:main1}
A right-angled Coxeter group has Property $\nl$ if and only if its defining graph $\Ga$ is complete.
\end{theorem}

Problem~\ref{ques1} is inspired by a large body of research on Coxeter groups aimed at understanding how properties of a Coxeter group can be discerned from the structure of the defining graph. We also note that most of the groups classically studied in geometric group theory, such as mapping class groups, right-angled Artin groups, and $\op{Out}(\mathbb{F}_n)$, fail to have Property $\nl$. However, there exist finite Coxeter groups, which naturally have Property $\nl$. Problem~\ref{ques1} is thus relevant for (infinite) Coxeter groups. A consequence of Theorem~\ref{thm:main1} is the following equivalent characterization of Property $\nl.$ 

\begin{cor} A right-angled Coxeter group has Property $\nl$ if and only if it is finite.    
\end{cor}

In the non-right-angled case, we limit our focus to triangle groups, which are Coxeter groups whose defining graphs have only three vertices. In this case, there are no restrictions on edge weights. There are three distinct types of triangle groups: spherical, euclidean, and hyperbolic (see section~\ref{sec:nonracgs}). These terminologies derive from whether the triangle group results in a tiling of the sphere, the plane, or the hyperbolic plane, respectively \cite{magnus}. Theorem~\ref{thm:main2} may therefore be thought of as classification result for triangle groups. 

\begin{theorem}\label{thm:main2}
A triangle Coxeter group has Property $\nl$ if and only if it is not hyperbolic. 
\end{theorem}

In addition, the defining graph can be used to determine whether a triangle group is hyperbolic, which provides a positive answer to Question~\ref{ques1} in this case also.

\medskip\noindent\textbf{Paper organization.} In Section~\ref{sec:background}, we introduce some background on Coxeter groups and on isometric group actions on hyperbolic spaces. In Section~\ref{sec:racgs}, we prove Theorem~\ref{thm:main1} about right-angled Coxeter groups. In Section~\ref{sec:nonracgs}, we prove Theorem~\ref{thm:main2} about triangle groups.

\medskip\noindent\textbf{Acknowledgements.} The authors would like to thank the summer REU program held at Georgia Institute of Technology in 2023 for bringing them all together to work on topics of common interest. A special thanks to Dan Margalit for organising and supervising all the projects conducted in geometric group theory of the CUBE group of the REU, and for providing helpful feedback on drafts of this paper.  Lastly, we thank Wade Bloomquist, Ryan Dickmann, and Abdoul Karim Sane for further mentorship during the REU program. 


\section{Background}\label{sec:background}

In this section, we collect some basic information about the terms used throughout the remainder of the paper. 

\subsection{Group Actions}
Describing the action of a group on a set is  one way of understanding the symmetries of the set being acted on. We can view an action as assigning each element of the group to a symmetry of the set. Formally, we have the following. 

\begin{definition} An action of a group $G$ on a set $X$ is a map $\varphi: G \times X \to X$ such that \begin{enumerate}
    \item $\varphi (e,x) = x$ for all $x \in X$ where $e$ is the identity element of $G$; \item $\varphi(g,\varphi(h,x)) = \varphi (gh, x)$ for all $g,h \in G$ and for all $x \in X$.
\end{enumerate}
\end{definition}

We will often use the shorthand notations $g\cdot x = \varphi(g,x)$ or $gx=\varphi(g,x)$ and denote by $G \acts X$ the action of $G$ on $X$. A group can always act on any space trivially, where every element of the group sends each element of the set to itself, or $g\cdot x = x$ for all $g \in G$ and for all $x \in X$. Every group can also act on itself by left multiplication via the group operation. In this paper, we will most often deal with non-trivial actions that satisfy certain properties.

\begin{definition} If $G \acts X$ and $X$ is a metric space with a metric (distance function) $d$, then the group action is called \emph{isometric} if $d(g\cdot x, g\cdot y) = d(x,y)$ for all $g \in G$ and all $x, y \in X$.
\end{definition}

Group actions are more interesting when the set is related to the group. For example, $S_n$ acts on the set $\{1, ... , n\}$ and, similarly, $D_n$, the group of symmetries of the regular $n$-gon, can be viewed as permuting the set of vertices of the associated polygon. 

\subsection{Hyperbolic Spaces}\label{subsec:hyp}
The hyperbolic plane differs from the Euclidean plane in how distances are measured and how geodesics (or paths of shortest distance between points) are defined. The motivation to define hyperbolic spaces comes from the Poincare disk model of the hyperbolic plane, denoted $\hh^2$. The hyperbolic plane $\hh^2$ can be visualized as a disk with open boundary that is conceptualized as a circle at infinity. Geodesics in this space tend to veer towards the interior of the circle rather than follow the straight line (in the usual Euclidean sense) between their endpoints.

Hyperbolicity was first defined by Gromov in terms of the \emph{Gromov product} in \cite{Gromov}; however, we employ an equivalent definition of hyperbolicity due to Rips.

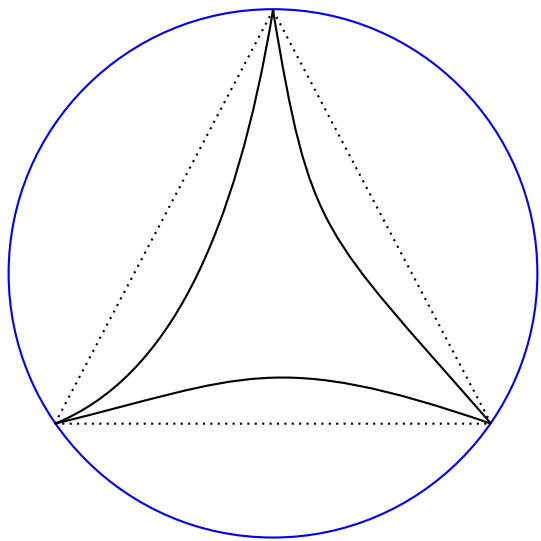
\begin{figure}[ht]
\centering
\begin{tikzpicture}
\draw[thick,blue] (0,0)  circle (100pt);
\draw[thick, dotted] (-2.9,-2) -- (0,3.5);
\draw[thick, dotted] (0,3.5) -- (2.9,-2);
\draw[thick, dotted] (-2.9,-2) --(2.9,-2);
\draw[thick] (-2.9,-2) .. controls (-1,-1,) and (-0.5,0.5) .. (0,3.5);
\draw[thick](2.9,-2) .. controls (1,1,) and (0.5,0.5) .. (0,3.5);
\draw[thick] (-2.9,-2) .. controls (0,-1,) and (0,-1) .. (2.9,-2);

\end{tikzpicture}
\caption{Geodesics in $\hh^2$ (solid) vs Euclidean geodesics (dotted).}
\label{fig:geodsinh2}
\end{figure}

\begin{definition} Let $(X,d)$ be a geodesic metric space. Let $\mathcal{N}_\e(\cdot)$ denote the $\e$-neighborhood of the argument. The space $X$ is \emph{hyperbolic} if there exists a $\delta \geq 0$ such that geodesic triangles are $\delta-$thin, i.e. for any $x, y, z \in X$ and any geodesics $[x,y], [y,z]$ and $[x,z]$ between them, $$\mathcal{N}_\delta( [x,y] \cup [y,z]) \supseteq [x,z].$$

The triangle with sides $[x,y], [y,z]$ and $[x,z]$ is called a \emph{geodesic triangle} (see Figure~\ref{fig:thintriangles}). The minimal constant $\delta$ that satisfies this condition is called the \emph{hyperbolicity constant} of $X$.    
\end{definition}

For example, trees (geodesic spaces with no closed loops) are $0-$hyperbolic since geodesic triangles are degenerate, forming tripods (see Figures~\ref{fig:linehyp} and ~\ref{fig:treehyp}). The hyperbolic plane $\hh^2$ is $\ln(1 + \sqrt{2})-$hyperbolic. 

\begin{figure}[ht]
\centering
\begin{tikzpicture}[font=\sffamily]
 \path (0,0) coordinate (A) (4,1) coordinate (B) (2,-2) coordinate (C);
 \draw[thick,path picture={
 \foreach \X in {A,B,C}
 {\draw[line width=0.4pt] (\X) circle (1);}}] (A) node[left]{$x$} to[bend right=25] 
 (B) node[above right]{$y$} to[bend right=20] 
 (C) node[below]{$z$} to[bend right=25] cycle;

\draw[thick, dotted, red, rotate=280] (0,2) ellipse (25pt and 100pt);
\draw[thick, dotted, blue, rotate=310] (2,1.5) ellipse (25pt and 85pt);
\draw[thick, dotted, green, rotate=220] (0,2) ellipse (50pt and 60pt);

\node[black] (label) at (6.5,1,5) {$\delta-$nbhds};
 
\end{tikzpicture}
\caption{Thin triangles condition}
\label{fig:thintriangles}
\end{figure}
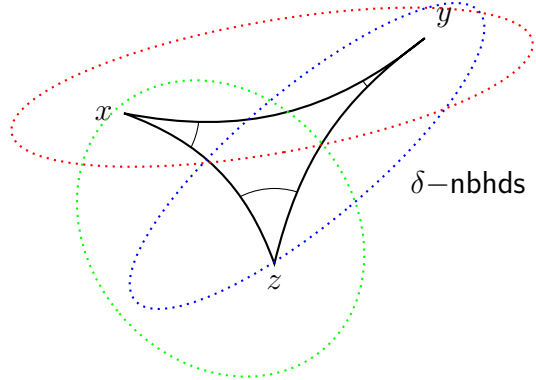

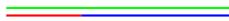
\begin{figure}[ht]
\centering
\begin{tikzpicture}[scale=0.5]
\draw[thick][red] (0,0) -- (-2,0);
\draw[thick][blue] (0,0) -- (4,0);
\draw[thick][green] (-2,0.2) -- (4,0.2);
\end{tikzpicture}
\caption{Degenerate geodesic triangles on a line.}
\label{fig:linehyp}
\end{figure}

\begin{figure}[ht]
\centering
\begin{tikzpicture}[scale=0.5]
\draw[thick][red] (0,0) -- (0,3);
\draw[thick][blue] (0,0) -- (-2,-2);
\draw[thick][green] (0,0) -- (2,-2);
\draw[thick] (0,3) -- (-2,5);
\draw[thick] (0,3) -- (2,5);
\draw[thick] (-2,-2) -- (-4,0);
\draw[thick] (-2,-2) -- (-4,-4);
\draw[thick] (2,-2) -- (4,0);
\draw[thick] (2,-2) -- (4,-4);
\end{tikzpicture}
\caption{Geodesic triangles in a tree are tripods.}
\label{fig:treehyp}
\end{figure}
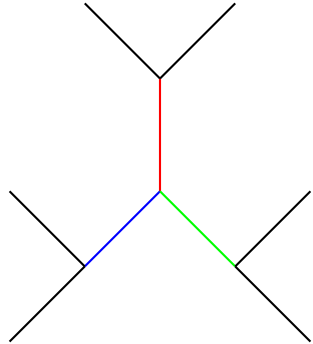

We now define the Gromov boundary of a hyperbolic space $X$, denoted $\partial X$, as follows. 

\begin{definition} Let $(X,d)$ be a hyperbolic metric space. Given points $x, y, z \in X$, the \emph{Gromov product} of points $x$ and $y$ with respect to $z$ is defined by
\[(x \mid y)_z = \dfrac{1}{2} \left( d(x,z) + d(y,z) - d(x,y) \right).\]

A sequence $(x_n)$ of points of $X$ \emph{converges at infinity} if $(x_i \mid x_j)_z \to \infty$
as $i, j \to \infty$ (this definition is independent of the choice of the base point $z$). Two such sequences $(x_i)$ and $(y_i)$ are \emph{equivalent} if $(x_i \mid y_j)_z \to \infty $ as $i, j \to \infty$. If $a$ is the equivalence class of $(x_i)$ under this equivalence, we say that the sequence $x_i$ converges to $a$. The boundary $\partial X$ is defined as the set of all such equivalence classes of sequences convergent at infinity. 
\end{definition}

We end this subsection with the definition of a quasi-isometry.

\begin{definition}
    Let $(X, d_X)$ and $(Y, d_Y)$ be metric spaces. A map $q \colon X \to Y$ is called a \emph{quasi-isometric embedding} if there exists a constant $C \geq 1$ such that $$ -C + \dfrac{1}{C} d_X( a,b) \leq d_Y ( q(a), q(b)) \leq C d_X(a,b) + C$$ for all $a,b \in X$. If, in addition, for every $y \in Y$, there is an $x \in X$ such that $d_Y(q(x), y) \leq C$, then $q$ is called a \emph{quasi-isometry}. 
\end{definition}

Note that quasi-isometry is an equivalence relation of the set of metric spaces. Further, if the metric spaces $X,Y$ are quasi-isometric, then $X$ is hyperbolic if and only if $Y$ is hyperbolic, i.e., quasi-isometries preserve hyperbolicity of metric spaces.

\subsection{Types of isometries in hyperbolic actions.}\label{subsec:isom}
If $G$ acts on a hyperbolic space by isometries, then we call such an action a \emph{hyperbolic action.} Gromov showed that for a hyperbolic group action $G \acts X$, an element $g \in G$ can be classified into one of three possible types of isometries.

\begin{enumerate}
    \item We say $g$ is \emph{elliptic} when the orbit of any $x_0 \in X$, is bounded, i.e. Given $x_0 \in X$, there exists a constant $C$ depdning on $x_0$ such that $d( g^n\cdot x_0, x_0) \leq C $ for all $n \in \nn$. In other words, repeatedly applying $g$ will not send $x_0$ to approach the boundary $\partial X$. 
    \item  We say $g$ is \emph{loxodromic} if the map $n \to g^n\cdot x_0$ is a quasi-isometric embedding $\zz \to X$. In this case, the boundary $\partial X$ contains two fixed points for $g$, $\xi^+$ and $\xi^-$. In this case, the action of $g$ appears similar to a translation, with $g^n \cdot x_0 \to \xi^+$ and $g^{-n}\cdot x_0 \to \xi^-$. 
    \item We say $g$ is \emph{parabolic} if it is neither elliptic nor loxodromic. In other words, the boundary $\partial X$ contains exactly one fixed point for $g$, say $\xi \in \partial X $. In this case, the orbit of $g$ is unbounded, and we have that both $g^n \cdot x_0 \to \xi$ and $g^{-n}\cdot x_0 \rightarrow \xi$.
\end{enumerate}

Property $\nl$, defined below, is characterized by the lack of loxodromic isometries. Loxodromic isometries have been very useful in studying dynamics of group actions and often yield a rich structure for the concerned groups; we refer to \cite{Gromov, ABO, Osin, dgo, bestvinasurvey} and references therein for the interested reader. 

\subsubsection{Infinite Dihedral Group}\label{subsec:dinf}
One important example of a group action used in the main result in this paper is the action of the infinite dihedral group, $D_{\infty}$, on the real line, which is a hyperbolic space. The infinite dihedral group is the complete group of symmetries of the line marked at the integral points; we define it as $D_{\infty}=\zz \rtimes (\zz \slash 2\zz).$

\begin{figure}[ht]
\centering
\begin{tikzpicture}[scale=0.5]
\draw[thick][<->] (-12,0) -- (12,0);
\draw[thick, dotted] (0,2) -- (0,-2);
\draw[thick, dotted] (4,2) -- (4,-2);

\node (zero) at (0.5,-0.5) {$0$};
\node (one) at (4.5, -0.5) {$1$};
\node (r) at (0,2.5) {$\leftrightarrow$ r};
\node (s) at (4,2.5) {$\leftrightarrow$ s};

\end{tikzpicture}
\caption{$D_\infty$ as a group of reflections.}
\label{fig:dinf}
\end{figure}
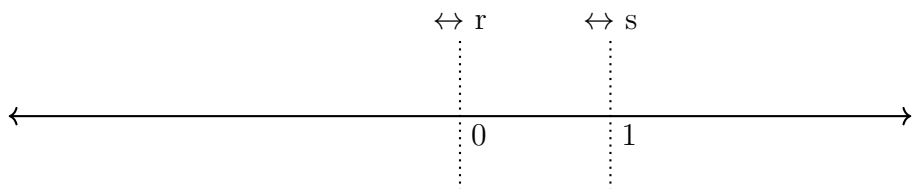

We can also view the infinite dihedral group as the group generated by two parallel reflections across different fixed points on the real line. Let $r$ be the reflection across the point $x=0$ where $r(x) = -x$ and let $s$ be the reflection across the point $x=1$ where $s(x) = 2-x$. Composing these two reflections yields a loxodromic element in the group, called a \emph{translation}. For example, performing $r$ and then $s$ results in a translation $g = sr$ to the right by two units. It follows that as $n \to \infty$, we have $g^n x_0 \to \infty$ and $g^{-n}x_0 \to -\infty$ for any $x_0 \in \zz$. We can see that $g$ fixes only the two points on the boundary of line, namely the points at $\infty$ and $-\infty$. Thus, it satisfies the definition of a loxodromic element. We can check that the generators $r,s$ do not commute by computing the element $rs$, which is a translation to the left by two units. Elements of $D_\infty$ are either of order two (reflections or, equivalently, \emph{involutions}) or are of infinite order (called \emph{translations}). This nomenclature takes inspiration from the action of these elements on the line.

 \subsection{Coxeter Groups}\label{subsec:cox}
Let $\Ga$ be a finite graph with $n \geq 1$ vertices $s_1, \cdots, s_n$ with the following information:  vertices $s_i, s_j,$ with $i \neq j$ may be connected by an edge labelled by a weight $m_{ij} = m_{ji} \in [2, \infty) \cap \zz$. $\Ga$ is then called a \emph{Coxeter graph}. 

Associated to such a Coxeter graph, one can define a \emph{Coxeter group} $W_\Ga$ as follows. A generating set for $W_\Ga$ is $S = \{s_1, s_2, \cdots s_n\}$. Moreover, $s_i^2=1$ and if two vertices $s_i$ and $s_j$ with $i \neq j$ are connected by an edge labeled by $m_{ij}$, then $(s_is_j)^{m_{ij}} =1.$

In other words, the following is a presentation of $W_\Ga$:

\begin{align*}
    W_\Ga = \langle s_1,\ldots,s_m | & s_1^2=\ldots=s_m^2 =1,\\
        & (s_is_j)^{m_{ij}} =1 \text{ if there is an edge labelled $m_{ij}$ between }s_i \text{ and } s_j \rangle. 
\end{align*}


 A missing edge denotes the \emph{free} relation between the associated generators; i.e. generators not connected by any edge lack any relation between them. A graph is called \emph{complete} if every pair of vertices has an edge between them. When $m=2,$ we treat the edges as unlabeled since the labeling is implied. The relation arising from $m=2$ is important because it indicates that $s_i$ commutes with $s_j$: since each $s_i$ is its own inverse, when $(s_is_j)^2 = s_is_js_is_j = 1 $, we can right multiply by $s_js_i$ to see that $s_is_j = s_js_i$. Right-angled Coxeter groups are a special case in which the only relations between generators are $(s_is_j)^{m_{ij}} = 1$ with $m_{ij} = 2$. In other words, in the right-angled setting, two generators must either commute with each other or have no relation. 

Returning to the previous example of the infinite dihedral group, we see that its associated Coxeter graph must have two vertices, as the group is generated by two reflections. Since their composition results in a translation (an element with infinite order), there is no edge connecting them; see Figure~\ref{fig:egsofcgs} for more examples of Coxeter groups.

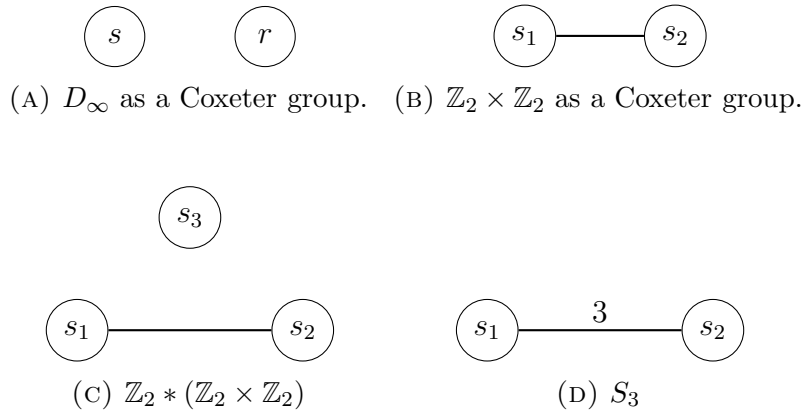
\begin{figure}[ht]

\begin{subfigure}{0.33\textwidth}
\centering
\begin{tikzpicture}[scale=0.5]
\node[circle, draw, minimum size=0.8cm] (s) at (-2, 0) {$s$};
\node[circle, draw, minimum size=0.8cm] (r) at (2, 0) {$r$};
\end{tikzpicture}
\caption{$D_\infty$ as a Coxeter group.}
\end{subfigure}%
\begin{subfigure}{0.33\textwidth}
\centering
\begin{tikzpicture}[scale=0.5]
\node[circle, draw, minimum size=0.8cm] (s_1) at (-2, 0) {$s_1$};
\node[circle, draw, minimum size=0.8cm] (s_2) at (2, 0) {$s_2$};
\draw[thick] (s_1)--(s_2);
\end{tikzpicture}
\caption{$\zz_2 \times \zz_2$ as a Coxeter group.}
\end{subfigure}

\vspace{25pt}

\begin{subfigure}{0.33\textwidth}
\centering
\begin{tikzpicture}[scale=0.5]
\node[circle, draw, minimum size=0.8cm] (s1) at (-3, 0) {$s_1$};
\node[circle, draw, minimum size=0.8cm] (s2) at (3, 0) {$s_2$};
\draw[thick] (s1)--(s2);
\node[circle, draw, minimum size=0.8cm] (s3) at (0, 3) {$s_3$};
\end{tikzpicture}
\caption{$\zz_2 \ast (\zz_2 \times \zz_2)$}
\end{subfigure}%
\begin{subfigure}{0.33\textwidth}
\centering
\begin{tikzpicture}[scale=0.5]
\node[circle, draw, minimum size=0.8cm] (s) at (-3, 0) {$s_1$};
\node[circle, draw, minimum size=0.8cm] (r) at (3, 0) {$s_2$};
\draw[thick] (s)--(r);
\node (e1) at (0,0.5) {$3$};
\end{tikzpicture}
\caption{$S_3$}
\end{subfigure}
\caption{Examples of Coxeter graphs and their corresponding Coxeter groups.}
\label{fig:egsofcgs}
\end{figure}

\subsection{Property (NL)}\label{subsec:propnl}
The formal study of Property $\nl$ was initiated by Balasubramanya, Fournier-Facio, and Genevois in \cite{PropNL}.

\begin{definition}
 We say that $G$ has Property $\nl$--- which stands for ``no loxodromics" ---  when for any action of $G$ on a hyperbolic space $X$, $G$ only contains elliptic or parabolic isometries.
\end{definition} 

We will sometimes simply say that a group $G$ has $\nl$ when we mean that $G$ has Property $\nl$. Finite and torsion groups are examples of groups with this property, since each element in such groups has finite order (and loxodromic elements must have infinite order). 

In addition to exploring the stability of Property $\nl$ under different group operations, the authors of \cite{PropNL} also produced many new examples of groups with the property. In particular, many so-called Thompson-like groups have Property $\nl$ (see \cite[Theorem 1.6]{PropNL} for a list of examples).  

We note the following results about Property $\nl$, which will be used crucially in this paper. The first concerns quotients of groups. 

\begin{prop}\cite[Lemma 4.2]{PropNL}\label{prop:nlquot} Let $N$ be a normal subgroup of $G$ and $\phi \colon G\to G/N$ be the quotient map. If $G/N \acts X$ is a hyperbolic action, then the action of $G \acts X$ defined by $g \cdot x = \phi(g) \cdot x$ is a well defined hyperbolic action of the same type. In particular, if $G/N \acts X$ contains a loxodromic element, then so does $G \acts X$.
\end{prop}

The next result follows from \cite[Corollary 1.5]{PropNL}. 

\begin{prop}\label{prop:nlamen}
Let $G$ be a finitely generated virtually abelian group. Then $G$ has Property $\nl$ if and only if it
does not surject onto $\zz$ or $D_\infty$.
\end{prop}

The motivation to study Property $\nl$ in Coxeter groups comes from the following observations: many classically studied groups in geometric group theory admit many actions on hyperbolic spaces with loxodromic elements, as proved by Abbott, Balasubramnya, and Osin in \cite{ABO}. 

In particular, right-angled Artin groups (RAAGs) never have Property $\nl$. Indeed, if $G$ is a RAAG that does not decompose as a direct product, then it admits an action with a loxodromic element on the associated \emph{extension graph}, which is a hyperbolic space. This was proved by Kim and Koberda in \cite{KimKoberda}. Otherwise, $G = A \times B$ for two infinite RAAGs $A,B$, where are least one of $A,B$ does not decompose as a direct product. Suppose $A$ does not decompose as a direct product. Then we may take the quotient $G/B \cong A$ and construct a hyperbolic action of $G$ with a loxodromic element by using Proposition~\ref{prop:nlquot}. 

In contrast, right-angled Coxeter groups (RACGs) can be finite groups, which have Property $\nl$. It is a natural question to consider which infinite RACGs, and more generally, which Coxeter groups, have this property. An interesting area of research in geometric group theory aims at understanding how a property of a RACG is related to the structure of its defining graph; we can therefore ask how to identify whether a RACG has Property $\nl$ using the defining graph.

\section{Right-Angled Coxeter Groups and Property $\nl$}\label{sec:racgs}
Our first result fully classifies which right-angled Coxeter groups have Property $\nl$. 

\begin{theorem}
    The right-angled Coxeter group $W_\Ga$ has Property $\nl$ if and only if the defining graph $\Ga$ is complete.
\end{theorem}

\begin{proof}
    We first prove that if $\Ga$ is a complete right-angled Coxeter graph, then $W_{\Ga}$ has $\nl$. Consider the complete right-angled Coxeter graph on $n$ vertices, $s_1, ... , s_n$. Since the graph is complete, each $s_i$ commutes with every $s_j$ for $1 \leq i,j \leq n$. Any element in $W_{\Ga}$ can thus be expressed in the normal form $s_1^{m_1}...s_n^{m_n}$ with $m_k \in \nn$ for $1 \leq k \leq n$. In addition, since $s_i^2 = 1$ for all $1 \leq i \leq n$, we can further simplify this element by specifying that  $m_k \in \{0,1\}$ for $1 \leq k \leq n$. In other words, $s_i^{m_i}$ is equal to the identity when $m_i = 0$ and is equal to $s_i$ when $m_i = 1$. Therefore, $W_{\Ga}$ is isomorphic to ${\zz}_2 \times \cdots \times {\zz}_2$ for $n$ copies of ${\zz}_2$. Since this is a finite group, we conclude that $W_{\Ga}$ has Property $\nl$.

We now prove the converse via the contrapositive. We will show that if $\Ga$ is not complete, then $W_{\Ga}$ must contain a loxodromic element. 

Since $\Ga$ is not complete, there exist two vertices, say $s_1$ and $s_2$, that are not connected by an edge. Thus for $s_1s_2 \in W_{\Ga}$, it follows that $ |s_1s_2| = \infty$. 

Recall the presentation of $D_{\infty}$ as the Coxeter graph with two disconnected vertices, $s$ and $r$ (see Figure~\ref{fig:egsofcgs}(A)). We define a map $\varphi: W_{\Ga} \to D_{\infty}$ by $\varphi(s_1) = s$, $\varphi(s_2) = r$, and $\varphi(s_i) = 1$ for $i \neq 1,2$. We extend the map $\varphi$ to any $w \in W_{\Ga}$ where $w$ is the product of generators $$w = s_{k_1}s_{k_2}s_{k_3}...s_{k_m}$$ for $k_i \in \{1, \cdots, n \}$ by  the rule $$\varphi(w) = \varphi(s_{k_1})\varphi(s_{k_2})\varphi(s_{k_3}) \cdots \varphi(s_{k_m}).$$  

Since our map $\varphi$ is extended canonically over generators, we have that $\varphi$ is a homomorphism. We claim that $\varphi$ is surjective. Consider the element $(sr)^p s^q$ where $q \in \{0,1\}$ and $p \in \zz$. We have that $$\varphi((s_1s_2)^p s_1^q) = (sr)^p s^q.$$ Similarly, elements of the form $(rs)^p r^q$ also have well-defined preimages in $W_\Ga$. Thus $\varphi$ is surjective.

By the first isomorphism theorem, $W_{\Ga}/ \ker\varphi \cong D_{\infty}$. As $D_\infty$ admits a hyperbolic action with a loxodromic element, so does $W_{\Ga}/ \ker\varphi$. By Proposition~\ref{prop:nlquot}, it follows that $W_{\Ga}$ does not have $\nl$, as desired.
\end{proof}

\section{Other Coxeter Groups}\label{sec:nonracgs}

In this section, we consider Coxeter groups that are not right-angled. A complete classification of which groups have Property $\nl$ in this more general setting seems to be hard to tackle. However, in the case where the defining graph is disconnected, we have a concrete answer. 

\begin{prop} If $\Ga$ is a disconnected Coxeter graph, then the Coxeter group $W_\Ga$ does not have Property $\nl$.
\end{prop}

\begin{proof}
Let $\Ga$ be a disconnected Coxeter graph on $n$ vertices, $s_1,...,s_n$, with non-empty connected components, $\Ga_1, \Ga_2,...,\Ga_m$. Let $X_t$ be the set of generators for $W_{\Ga_t}$. Recall that $D_\infty$ is generated by two reflections $r$ and $s$. 

We define a map $\phi: W_{\Ga} \rightarrow D_{\infty}$ as follows 
\begin{align*}
\varphi(s_h) &= r \text{ for all } x_h \in X_1,\\
\varphi(s_i) &= s \text{ for all } x_i \in X_2 \text{ and },\\
\varphi(s_j) &= 1 \text{ for all } x_j \in X_3\cup ... \cup X_m,\\
& \text{ and extend $\varphi$ to $W_\Ga$ as a homomorphism. }
\end{align*}

 Note that if $a,b \in X_1$ (or $X_2$) are vertices connected by an edge labelled $p$, then we have $(ab)^p =1$ in $W_\Ga$. Under the map $\varphi$, this element is sent to $(r^2)^p = (1)^p =1$. Further, there are no relations between elements of $X_1$ and elements of $X_2$. Therefore, the map $\varphi$ does not introduce any new relations and is thus a homomorphism.

We now show that $\varphi$ is surjective. Consider the element $(rs)^m r^k \in D_{\infty}$ where $m \in \zz$ and $k \in \{0, 1\}$. Take any $s_h \in X_1$ and $s_i \in X_2$; then we have that $\varphi((s_hs_i)^mx_h^k)=(rs)^m r^k$. By the first isomorphism theorem,
\[W_{\Ga}/ \ker\varphi \cong D_{\infty}.\]
By Proposition~\ref{prop:nlquot}, we conclude that $W_\Ga$ does not have Property $\nl$.
\end{proof}

The above result limits our search for further Coxeter groups with Property $\nl$ to those with a connected defining graph. As a starting point of this exploration, we focus on Coxeter groups with 2 and 3 vertices. These groups are also known as dihedral groups and triangle groups, respectively.


\begin{prop}
Let $W_{\Gamma}$ be a Coxeter group with exactly 2 generators, $s_1$ and $s_2$. $W_{\Gamma}$ has Property $\nl$ if and only if its defining graph $\Gamma$ is complete.
\begin{proof}
Suppose the defining graph of $W_{\Gamma}$ is complete with edge label $m$. Because $W_{\Gamma}$ only has two generators, all group  elements can be represented as words of the form $(s_1s_2)^p r$ or $(s_2s_1)^q$ for some $p,q \in \zz$. Because $s_1s_2$ has order $m$, there are finitely many possibilities for $p$ and $q$. Therefore, $W_{\Gamma}$ is finite and has Property $\nl$.

The only Coxeter group with 2 vertices and an incomplete defining graph is $D_{\infty}$, which does not have Property $\nl$.
\end{proof}
\end{prop}

\begin{definition}
    A triangle group is a Coxeter group whose defining graph has three vertices, and three edges. Each edge is labelled by $l,m,n$, where $l,m,n \geq 2$ are integers. 
\end{definition}

Triangle groups can be visualized as groups generated by reflections across each of the three sides of a triangle. The resulting group is then the group of symmetries of a tiling by these triangles of one of three geometries: the sphere, Euclidean plane, or hyperbolic plane. This justifies the nomenclature of triangle groups. 

The order of the relation between two of these generating reflections is closely related to the angle between the associated sides of the triangle. Consider a right triangle in the Euclidean plane. Reflecting across one of the legs and then the perpendicular leg will result in a rotation by $\pi$, which is its own inverse, and thus is an order 2 element. This is also why we call those Coxeter groups with only weights 2 on the edges right-angled. 

As explained above, we can see that in the right-angled case, when the reflections intersect at an angle of $\frac{\pi}{2}$, the order of the Coxeter relation between them is $2$. Similarly, a tiling of the Euclidean plane with equilateral triangles has generating reflections that intersect at an angle of $\frac{\pi}{3}$. It is straightforward to check that the composition of two of these reflections three times will yield the original triangle, showing that these generators have a Coxeter relation of order $3$. 

More generally, given a Coxeter graph with three vertices and edge relations labeled $l,m,n$, the corresponding triangle has interior angles $\dfrac{\pi}{l}, \dfrac{\pi}{m}$, and $\dfrac{\pi}{n}$. Only certain choices of $l,m,n$ will allow for Euclidean triangles, whose interior angles  must add up to $\pi$. In fact, we can  fully classify  triangle groups by the sum of their interior angles. Let $s = \dfrac{\pi}{l} +\dfrac{\pi}{m} +\dfrac{\pi}{n}$. Then,

\begin{enumerate}
    \item[(1)] $s = \pi$ corresponds to a triangle tiling of the Euclidean plane,
    \item[(2)] $s > \pi$ corresponds to a triangle tiling of a sphere, and
    \item[(3)] $s < \pi$ corresponds to a triangle tiling of the hyperbolic plane.
\end{enumerate}

We will consider each of these types of triangle groups individually and deduce in each case if the group has Property $\nl$ or not. In what follows, we will use the notation $\Delta(l,m,n)$ to denote the triangle group with edge labels $l,m,n$. We start with the spherical and Euclidean cases.

\begin{prop}
Spherical triangle groups have Property $\nl$.
\end{prop}

\begin{proof} In the spherical case, there are only finitely many triangles in the tiling of the sphere. Thus, the group of symmetries of this tiling is a finite group. As we have seen, all finite groups have Property $\nl$, so spherical triangle groups must have Property $\nl$. 
\end{proof}

\begin{lemma}\label{th:lemma}
No triangle group surjects onto the integers or $D_{\infty}$. 
\end{lemma}

\begin{proof}
We first claim that no triangle group surjects onto the integers. The only element of finite order in the group of the integers is $0$ (the identity). In order for a triangle group to map onto the integers, elements of finite order must be sent to elements of finite order. Thus, all generators of the triangle group would have to be sent to the identity, and so the map would be the trivial homomorphism. We conclude that no triangle group can surject onto the integers.

Now consider some triangle group $\Delta (l,m,n)$ with generators $a$, $b$, and $c$. By way of contradiction, assume there exists some surjective homomorphism $\phi$ from $\Delta (l,m,n)$ to $D_{\infty}$. Observe that for this homomorphism to be surjective, at least two of the generators must map to distinct non-trivial involutions in $D_{\infty}$. Otherwise the image of the homomorphism would be trivial or a finite subgroup of $D_{\infty}$ with only 2 elements, which implies $\phi$ is not surjective.

Recall $D_\infty$ is generated by the involutions $r$ and $s$. Then involutions in $D_{\infty}$ are either of the form $(rs)^pr$ or $(sr)^qs$ for some $p,q\in \nn$. 

We claim that the images of at least two generators under $\phi$ have to be distinct involutions. Let $x$ be a reflection in $D_{\infty}$ and suppose $\phi(a)=\phi(b)=\phi(c)=x$ or $\phi(a)=\phi(b)=x$ and $\phi(c)=\id.$ Then $\phi\left(\Delta(l,m,n)\right)$ is finite and therefore not equal to $D_{\infty}.$
Without loss of generality, let $a,b$ be the two generators whose images are distinct involutions.

We now have two cases. In the first case, $\phi(a) = (rs)^pr$  and $\phi(b) = (sr)^qs$ for some $p, q \in \nn$. So, $\phi(ab) = (rs)^{p+q+1}$. As $a,b$ are connected by an edge labelled by one of $
l,m,n$, $ab$ is a finite order element in $ \Delta(l,m,n)$. But then $(rs)^{p+q+1}$ has finite order in $D_{\infty}$. This is a contradiction because $rs$ is a translation and has infinite order.

In the second case, $a$ and $b$ are mapped to involutions that both start with the same generator. So without loss of generality, suppose $\phi(a) = (rs)^pr$ and $\phi(b) = (rs)^q r$ for some $p,q \in \nn$. If $p>q$, $\phi(ab)=(rs)^{p-q}$ and if $q >p$, $\phi(ab)=(sr)^{q-p}$. In either instance we obtain a contradiction as $ab$ has finite order, but $(sr)^{-1} = rs$ has infinite order.
So no surjective homomorphisms can exist from the triangle groups to $D_{\infty}$.
\end{proof}

\begin{prop}
Euclidean triangle groups have Property $\nl$.
\end{prop}

\begin{proof}
Unlike the spherical case, the triangle tiling of the plane extends infinitely, so the associated triangle group is not finite.
However, there are only three triangle groups that satisfy condition (1) above for the Euclidean case: $\Delta (3,3,3)$, $\Delta (2,4,4)$ and $\Delta (2,3,6)$. These are all well-documented in the classification of Coxeter groups, see for example \cite{trianglegrps}.

All three Euclidean triangle groups are virtually abelian of finite rank (in fact virtually $\zz^2$). Indeed, since these groups admit proper and cobounded actions on $\zz^2$, it follows from the Svarc--Milnor Lemma that they are all quasi-isometric to $\zz^2$ and thus contain a finite index subgroup isomorphic to $\zz^2$ \cite{qitoz2}. (See also \cite{etg} for a more precise description of the structure of these groups). By Lemma ~\ref{th:lemma}, none of these groups surject onto $D_{\infty}$ or $\zz$. Therefore, by Proposition~\ref{prop:nlamen}, the Euclidean triangle groups have Property (NL).
\end{proof}

\begin{prop}
Hyperbolic triangle groups do not have Property $\nl$.
\end{prop}

\begin{proof}
By definition, a hyperbolic Coxeter group $G$ tiles the hyperbolic plane, and therefore the action on $\hh^2$ is cobounded (the fundamental domain is given by the triangle with angles $\pi/l , \pi/m, \pi/n$). Furthermore, for any bounded set $B \in \hh^2$, we have that $\{ g | gB \cap B \neq \emptyset \}$ is a finite set, so the action is proper. By the Svarc--Milnor Lemma, there is a finite generating set $Y$ of $G$ such that the Cayley graph $\Ga (G,Y)$ is quasi-isometric to $\hh^2$. Since hyperbolicity is preserved under quasi-isometries, we have that  $\Ga (G, Y)$ is hyperbolic and thus $G$ is an infinite hyperbolic group. Therefore $G$ does not have Property $\nl$.
\end{proof}

\bibliographystyle{alpha}
\bibliography{bibliography}
\end{document}